\documentclass{elsarticle}
\usepackage{hyperref}
\usepackage{amsmath}
\usepackage{amssymb}
\usepackage{amsthm}
\usepackage{graphicx}
\usepackage{array}
\usepackage{ifpdf}
\ifpdf
\usepackage[all,pdf,2cell]{xy}\UseAllTwocells\SilentMatrices
\else
\usepackage[all,xdvi,2cell]{xy}\UseAllTwocells\SilentMatrices
\fi
\newcolumntype{x}{>{\centering\arraybackslash}m{3em}}

\newcommand{\lin}{\ell}

\newcommand{\Zn}{\mathbb Z_n}
\DeclareMathOperator{\Ker}{Ker}
\newcommand{\Reg}{\mathcal O}
\newcommand{\hatReg}{\hat{\mathcal O}}
\title{Linear extensions and order-preserving poset partitions}
\author[cor1]{Gejza Jen\v ca}
\cortext[cor1]{Corresponding author}
\ead{gejza.jenca@stuba.sk}
\author{Peter Sarkoci}
\ead{peter.sarkoci@stuba.sk}
\address{
Department of Mathematics and Descriptive Geometry\\
Faculty of Civil Engineering\\
Slovak Technical University\\
Radlinsk\' eho 11\\
	Bratislava 813 68\\
	Slovak Republic
}
\newtheorem{theorem}{Theorem}[section]
\newtheorem{corollary}[theorem]{Corollary}
\newtheorem{proposition}[theorem]{Proposition}
\newtheorem{lemma}[theorem]{Lemma}

\theoremstyle{definition}
\newtheorem{definition}[theorem]{Definition}
\newtheorem{example}[theorem]{Example}
\begin{document}
\begin{abstract}
We examine the lattice of all order congruences
of a finite poset from the viewpoint of combinatorial algebraic topology.
We prove that the order complex of the lattice of all nontrivial
order congruences (or order-preserving partitions) of a finite $n$-element
poset $P$ with $n\geq 3$ is homotopy equivalent to a wedge of spheres of dimension $n-3$. 
If $P$ is connected, then the number of spheres is equal to the number of
linear extensions of $P$. In general, the number of spheres is equal to the
number of cyclic classes of linear extensions of $P$. 
\end{abstract}
\maketitle
\section{Introduction}

An order congruence of a poset $P$ can be defined as a kernel
of an order-preserving map with domain $P$. Even if this notion is
simple and natural, the amount of papers dealing with it appears to be relatively
small. The notion appears in the seventies in a series of papers
by T. Sturm \cite{Stu:VvKIA,Stu:ECvK,Stu:OLoKoIM}, the same notion with a different
formulation appears in the W.T. Trotter's book \cite{Tro:CaPOSDT}.
A related notion in the area of ordered algebras appeared in
two papers by G. Cz\'edli a A. Lenkehegyi \cite{CzeLen:OCnDoOA,CzeLen:OCoOAaQD}.
In our approach, we will follow a recent paper by 
P. K\"ortesi, S. Radeleczki and S. Szil\'agyi \cite{KorRadSzi:CaIMoPOS}.

In the present paper we will examine the lattice of all order congruences
of a finite poset from the viewpoint of combinatorial algebraic topology.
We will prove that the order complex of the lattice of all nontrivial
order congruences (or order-preserving partitions) of a finite $n$-element
poset $P$ with $n\geq 3$ is homotopy equivalent to a wedge of spheres of dimension $n-3$. 

If $P$ is connected, then the number of spheres is equal to the number of
linear extensions of $P$. In general, the number of spheres is equal to the
number of cyclic classes of linear extensions (Definition \ref{def:ce}) of $P$. 

\section{Preliminaries}

We assume familiarity with standard notions of algebraic topology and
poset theory. See, 
for example, the books \cite{Mun:EoAT} or \cite{May:CCiAT} for topological terminology
and Chapter 3 of \cite{Sta:EC} or \cite{Gra:GLT} for poset terminology. See also
\cite{Bjo:TM} and \cite{Koz:CAT} for a survey of topological methods in combinatorics.

If a poset has top element, this is denoted by $\hat 1$.
The bottom element of a poset is denoted by $\hat 0$.

Let $P$ be a finite poset with $n$ elements. 
A {\em linear extension} of $P$ is an order-preserving bijection $f\colon P\to\{0,\dots,n-1\}$, where
the codomain is ordered in the usual way. For our purposes, this definition is more
appropriate than the standard one (see \ref{Sta:EC}, Section 3.5). 
The set of all linear extensions is denoted
by $\ell(P)$. The number of linear extensions of $P$ is denoted by $e(P)$.

Alternatively, we may view linear extensions as particular words over the alphabet $P$.
If $f$ is a linear extension on an $n$-element poset $P$, we sometimes identify $f$ with the word
$f^{-1}(0)f^{-1}(1)\dots f^{-1}(n-1)$.

For elements $x,y$ of a poset, we say that $x$ covers $y$ if
$x\geq y$, $x\neq y$ and for every element $z$ such that $x\geq z\geq y$
we have either $z=x$ or $z=y$.
The covering relation is denoted by $\succ$, $\prec$ denotes the inverse of $\succ$.

For a finite poset $(P,\leq)$, we write $\Delta(P)$ for the abstract
simplicial complex consisting of all chains of $P$, including the empty set.
If a finite poset $P$ has an upper or lower bound, then
$\Delta(P)$ is topologically trivial, that means, it is 
homotopy equivalent to a point. Thus, when
dealing with posets from the point of view of algebraic topology, 
it is usual (and useful) to remove bounds from a poset before applying
$\Delta$. If $P$ is a poset, then $\hat P$ denotes the same poset
minus upper or lower bounds, if it has any.

The {\em face poset} of a finite abstract simplicial complex $\Delta$ is
the poset of all faces of $\Delta$, ordered by inclusion. It is 
denoted by $\mathcal F(\Delta)$. 

\begin{definition}
Let $P$ be a finite poset. An {\em acyclic matching} on $P$ is a set 
$M\subseteq P\times P$
such that the following conditions are satisfied.
\begin{enumerate}
\item For all $(a,b)\in M$, $a\succ b$.
\item Each $a\in P$ occurs in at most one element in $M$;
if $(a,b)\in M$ we write $a=u(b)$ and $b=d(a)$.
\item There does not exist a cycle 
$$
b_1\succ d(b_1)\prec b_2\succ d(b_2)\prec\dots\prec b_n\succ d(b_n)\prec b_1.
$$
\end{enumerate}
\end{definition}

When constructing acyclic matchings for posets, the following theorem is sometimes used to
make the induction step.

\begin{theorem}(\cite[Theorem~11.10]{Koz:CAT})\label{thm:acyclicfibers} 
Let $P$ be a finite poset.
Let $\varphi\colon P\to Q$ be an order-preserving or an order-inverting mapping 
and assume that we have acyclic matchings on subposets
$\varphi^{-1}(q)$, for all $q\in Q$. Then the union of these acyclic matchings is itself an
acyclic matching on $P$.
\end{theorem}

In the context of Theorem \ref{thm:acyclicfibers}, the sets $\varphi^{-1}(q)$
are called the {\em fibers of $\varphi$}.

The following theorem shall be used in the proof of the main result.

\begin{theorem}(\cite[Theorem~6.3]{For:AUGtDMT})
\label{DMT}
Let $\Delta$ be a finite simplicial complex.
Let $M$ be an acyclic matching of the face poset of $\Delta$ such that
all faces of $\Delta$ are matched by $M$ except for $n$ unmatched faces of dimension $d$. 
Then $\Delta$ has
the homotopy type of the wedge of $n$ spheres of dimension $d$.
\end{theorem}

We remark that our wording of Theorem \ref{DMT} is slightly different than the original one,
since we allow the empty face of $\Delta$ to be matched.

\section{Order-preserving partitions}

The goal of the notion of an order preserving partition is to introduce
a notion of ``congruence'' on a poset. Obviously, we want this notion to behave
nicely, that means, we want the quotient of a poset to be
again a poset. Before we give
a formal definition, let us write informally one possible line of thoughts
leading to the (probably only possible) definition.

Let $(P,\leq)$ be a poset, let $\rho$ be an equivalence relation on $P$.
We write $P/\rho$ for the set of equivalence classes of $\rho$ and
$\rho[x]$ for the equivalence class of $x\in P$. A natural way of transferring
the relation $\leq$ from $P$ to $P/\rho$ is to write, for equivalence
classes $B_1,B_2$ of $\rho$, $B_1\leq B_2$ if and only if there are
$x\in B_1$, $y\in B_2$ such that $x\leq y$. 

In general, this relation on $P/\rho$ is not transitive. There is an easy remedy for 
this problem: we can take a transitive closure of the previous version of
$\leq$ and denote this new relation $\leq$, again. This leads us to the notion
of $\rho$-sequence (see Definition \ref{def:opp} below). In general, even this relation
$\leq$ on $P/\rho$ is not a partial order, since it need not to be
antisymmetric. It turns out that this problem is caused by $\rho$-cycles (see below)
that are not contained in an equivalence class of $\rho$. The only solution seems
to be to assume that there are no such problematic $\rho$-cycles; and this
is exactly the notion of {\em order preserving partition} given by the
following definition.

\begin{definition}[\cite{KorRadSzi:CaIMoPOS}]\label{def:opp}
Let $(P,\leq)$ be a poset and let $\rho\subseteq P^2$ be an equivalence relation
on it. 
\begin{enumerate}
\item[(i)] A sequence $x_0,\dots,x_n\in P$ is called a {\em $\rho$-sequence} if for
each $i\in\{1,\dots,n\}$ either $(x_{i-1},x_i)\in\rho$ or $x_{i-1}<x_i$ holds.
If in addition $x_0=x_n$, then $x_0,\dots,x_n$ is called a {\em $\rho$-cycle}.
\item[(ii)] We say that $\rho$ is an {\em order-congruence} of $(P,\leq)$ if for every
$\rho$-cycle $x_0,\dots,x_n\in P$, $\rho[x_0]=\dots=\rho[x_n]$ is satisfied.
\item[(iii)] A partition $\pi$ is called an {\em order-preserving partition} of
$(P,\leq)$ if $\pi=(P/\rho)$ for some order congruence $\rho$ of $(P,\leq)$.
We write $\pi=\pi_\rho$ or $\rho=\rho_\pi$.
\item[(iv)] If $\pi$ is an order-preserving partition we say that
a sequence $x_0,\dots,x_n$ is a {\em $\pi$-sequence} or a {\em $\pi$-cycle} if
$x_0,\dots,x_n$ is a $\rho_\pi$-sequence or a $\rho_\pi$-cycle, respectively.
\end{enumerate}
\end{definition}

\begin{lemma}[\cite{KorRadSzi:CaIMoPOS}]
If $\rho$ is an order-congruence of a poset $(P,\leq)$, then it induces a partial
order $\leq_\rho$ defined on the set $P/\rho$ as follows:

$\rho[x]\leq_\rho\rho[y]$ if there exists a $\rho$-sequence $x_0,\dots,x_n\in P$, with
$x_0=x$ and $x_n=y$.
\end{lemma}

In view of the previous lemma, we can consider $\pi_\rho$ as a poset with the partial
order $\leq_\rho$ determined by $\leq$. In what follows, we write simply $\leq$ 
instead of $\leq_\rho$, if there is no danger of confusion.

Recall, that a {\em quasiorder} (sometimes called a {\em preorder}) is 
a binary relation that is reflexive and transitive. For an order-preserving
map of posets $f:P\to Q$ we write
$$
\Ker f:=\{ (x,y)\in P\times P: f(x)=f(y)\}.
$$
The relation $\Ker f$ is called {\em the kernel} of $f$.

\begin{theorem}[\cite{CzeLen:OCnDoOA}]
Let $(P,\leq)$ be a poset and let $\rho$ be an equivalence relation on $P$.
Then the following are equivalent.
\begin{enumerate}
\item[(i)]The relation $\rho$ is an order-congruence of $(P,\leq)$.
\item[(ii)]There exists a poset $(Q,\leq)$ an an order-preserving map $f\colon P\to Q$ such that
$\rho=\Ker f$.
\item[(iii)]The relation $\leq$ can be extended to a quasiorder $\theta$ such that $\rho=\theta\cap\theta^{-1}$.
\end{enumerate}
\end{theorem}
\begin{figure}
\includegraphics{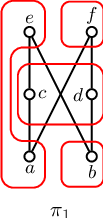}
\hskip 3em
\includegraphics{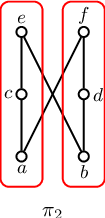}
\hskip 3em
\includegraphics{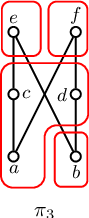}
\hskip 3em
\includegraphics{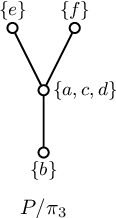}
\caption{The partition $\pi_3$ is order-preserving, whereas $\pi_1,\pi_2$ are not.}
\label{fig:nonreg}
\end{figure}

\begin{example}
Consider a 6-element poset $P$ and its three partitions $\pi_1,\pi_2,\pi_3$
as shown in Figure \ref{fig:nonreg}. 

The partition $\pi_1$ is not order-preserving, since $a,c,e,a$ is a $\pi_1$-cycle with $[a]_{\pi_1}\neq[c]_{\pi_1}$. 
In fact, it is easy to see that every
block of an order-preserving partition must be order-convex.

Although $\pi_2$ has only order-convex blocks, it fails to be order-preserving.
Indeed $a,f,b,e,a$ is a $\pi_2$-cycle with $[a]_{\pi_2}\neq[f]_{\pi_2}$.

Finally, $\pi_3$ is an order-preserving partition, the diagram of the quotient poset $P/\pi_3$ 
is shown in the picture. 
\end{example}

Let us consider the set $\Reg(P)$ of all order-preserving partitions of $P$
equipped with a partial order $\leq$ defined as
the usual refinement order of partitions:
$\pi_1\leq\pi_2$ iff every block of $\pi_1$ is a subset of a block of $\pi_2$.

The bottom element of
$\Reg(P)$ is the partition consisting of singletons,
the top element is the partition with a single
block.

The poset $\Reg(P)$ is an algebraic lattice \cite[Theorem~30]{Stu:OLoKoIM}. 
For order-preserving partitions
$\pi_1$, $\pi_2$
$$
\pi_1\wedge\pi_2=\{B_1\cap B_2: \text{$B_1\in\pi_1$, $B_2\in\pi_2$ and 
	$B_1\cap B_2\neq\emptyset$}\}.
$$

Let $\pi$ be an order-preserving partition of a poset $P$.
For $x,y\in P$, we write $x\lesssim_\pi y$ iff there is a $\rho_\pi$-sequence
from $x$ to $y$. Note that $\lesssim_\pi$ is a quasiorder and
that $\lesssim\_pi$ is simply the transitive closure of the relation 
$\leq\cup\rho_\pi$.

To define joins, we may proceed as follows.
Let $\pi_1,\pi_2\in\Reg(P)$
and $\lesssim$ 
be the transitive closure
of the union of $\lesssim_{\pi_1}$ and $\lesssim_{\pi_2}$.
Clearly, $\lesssim$ is a quasiorder on $P$. For $x,y\in P$,
write $x \sim y$ iff $x\lesssim y$ and $y\lesssim x$.
Then $P/\sim$ is an order-preserving partition of $P$ and
$\pi_1\vee\pi_2=(P/\sim)$

The covering relation in the lattice of order-preserving
partitions of a finite poset is easy to describe: for a pair $\pi_1,\pi_2$ 
of order-preserving partitions of a finite poset $P$ we have 
$\pi_1\prec\pi_2$ iff $\pi_2$ arises from $\pi_1$ by merging
of two blocks $B_1,B_2$ of $\pi_1$ such that 
\begin{itemize}
\item either $B_1\prec B_2$ in the poset $(\pi_1,\leq)$, or
\item $B_1$ and $B_2$ are incomparable in the poset $(\pi_1,\leq)$, in symbols $B_1\parallel B_2$.
\end{itemize}

In particular, this implies that the atoms of the lattice of order-preserving
partitions of a finite poset $P$ is the set of all partitions of $P$ that
are of the form
$$
\pi_{a,b}:=\{\{a,b\}\}\cup\{\{x\}:x\in P-\{a,b\}\},
$$
where $a,b\in P$ is such that either $a\prec b$ or $a\parallel b$.
Moreover, the lattice $\Reg(P)$ is ranked. 
The ranking function is given by $|P|-|\pi|$.

\begin{example}
\label{ex:antichain}
If $A_n$ is an $n$-element antichain, then every partition of $A_n$ is order-preserving.
The lattice of order-preserving partitions is then the partition lattice
of the set $A_n$, usually denoted by $\Pi_n$. It is well known 
\cite{Fol:ThGoaL,Koz:CAT}, that for all $n\ge 3$
the order complex of $\hat\Pi_n$ is homotopic to the wedge of 
$(n-1)!$ spheres of dimension $n-3$.
\end{example}

\begin{example}\label{ex:chain_ordcongr}
If $C_n$ is an $n$-element chain, $n\ge 3$, then a partition $\pi$ of $C_n$
is order-preserving if and only if all blocks of $\pi$ are convex subsets of $C_n$. It
is easy to see that $\Reg(C_n)$ is a Boolean algebra $B_{n-1}$ with $n-1$ atoms.
It is well known that the order complex of $\hat B_{n-1}$ is homotopic to a 
single sphere of dimension $n-3$.
\end{example}
\begin{figure}
\includegraphics{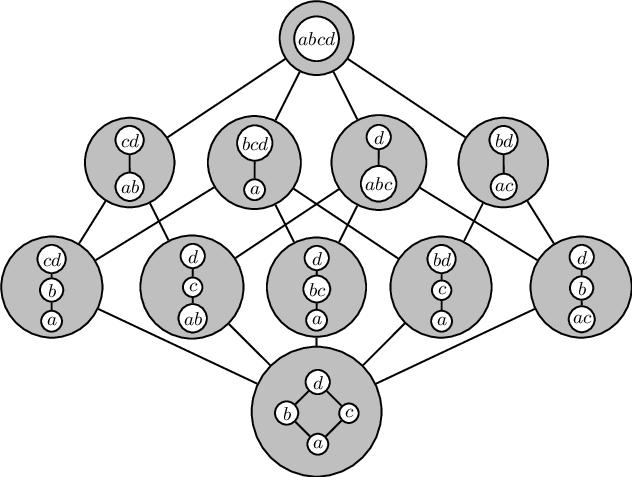}
\caption{Order-preserving partitions of a Boolean algebra with two atoms}
\label{fig:big}
\end{figure}

\begin{example}\label{ex:btwo}
To give a slightly more complicated example, let $B_2$ be a
Boolean algebra with two atoms. The lattice of order-preserving partitions
of $B_2$ has 11 elements; its Hasse diagram is Figure \ref{fig:big}. Note that
$\Delta(\hatReg(B_2))$ is not semimodular. 

It is easy to see that $\Delta(\hatReg(B_2))$ has the homotopy type
of a wedge of two spheres of dimension one.
\end{example}

The proof of the following Theorem is inspired by the proof
of Theorem 11.18 in \cite{Koz:CAT}, where the homotopy type of $\Delta(\hat\Pi_n)$
is determined.

\begin{theorem}
\label{thm:numberofspheres}
Let $P$ be a finite poset with $n$ elements, $n\geq 3$. Then $\Delta(\hatReg(P))$ is
homotopy equivalent to a wedge of spheres of dimension $n-3$. Let $a$ be
a minimal element of $P$. Write $s_O(P)$ for the number of
spheres in $\Delta(\hatReg(P))$. For $n>3$, $s_O(P)$ satisfies the recurrence
$$
s_O(P)=\sum_{\substack{b\in P\\\pi_{a,b}\text{ order-preserving}}}s_O(\pi_{a,b}).
$$
\end{theorem}
\begin{proof}
Using induction with respect to $n$, we shall prove that for every
finite poset $P$ with at least three elements, there exists an
acyclic matching on the face poset of $\Delta(\hatReg(P))$ such that all faces
are matched except for $s_O(P)$ faces of dimension $n-3$. By Theorem \ref{DMT},
this implies that $\Delta(\hatReg(P))$ is homotopy equivalent to a wedge of
$s_O(P)$ spheres of dimension $n-3$. The recurrence for $s_O(P)$ will be proved as
a byproduct of the induction step.

For every 3-element poset $P$, $\hatReg(P)$ is an antichain and it is obvious that there is a
matching as claimed.

Let $n>3$ and suppose that for any poset $Q$ with $|Q|=n-1$ there is an acyclic matching
on the face poset of $\Delta(\hatReg(Q))$ such that all faces
are matched except for $s_O(Q)$ faces of dimension $|Q|-3=n-4$.

Fix a minimal element $a$ of $P$. Let $P_a$ be the poset of
all order-preserving partitions containing $\{a\}$ as a singleton class, ordered by refinement.
Let $\pi_a=\{\{a\},P\setminus\{a\}\}$; it is clear that $\pi_a$ is an order-preserving partition of $P$
and that it is the top element of $P_a$.
Let $\phi\colon \mathcal F(\Delta(\hatReg(P)))\to P_a$ be given by the following rules:
\begin{itemize}
    \item if $c$ is a chain consisting solely of elements of $P_a$, then $\phi(c)=\pi_a$,
    \item otherwise let $\pi_{min}$ be the smallest element of $c$ such that $\pi_{min}\notin P_a$;
put $\phi(c)=\pi_{min}\wedge\pi_a$.
\end{itemize}
It is obvious that $\phi$ is an order-inverting mapping.
We shall construct acyclic matchings on the fibers of $\phi$. By Theorem \ref{thm:acyclicfibers},
the union of these matchings is an acyclic matching on 
$\mathcal F(\Delta(\hatReg(P)))$.

Let $S=\phi^{-1}(\pi)$ where $\pi$ is not the bottom element of $P_a$,
that means, the partition of $P$ into singletons.
Then we can construct the matching on $S$ by either removing or adding
$\pi$ from each chain, depending on whether it does or does
not contain $\pi$. 

Let $S=\phi^{-1}(\hat 0)$, where $\hat 0$ is the partition of $P$ into
singletons. This means, that for every chain $c\in S$ the smallest element $\pi_{min}$ 
of $c$ not belonging to $P_a$ must be such that $\pi_{min}\wedge\pi_a=\hat 0$.
This implies that $\pi_{min}$ has a single non-singleton class, in
other words, $\pi_{min}=\pi_{a,b}$ for some $b$.
Moreover, whenever $c\in\mathcal F(\Delta(\hatReg(P)))$ is such that
$\pi_{a,b}\in c$, then $c\in S$. Thus $S$ is the set of all 
$c\in\Delta(\hatReg(P))$ such that $\pi_{a,b}\in c$. Let us write
$$
S_{a,b}=\{c\in\mathcal F(\Delta(\hatReg(P))):\pi_{a,b}\in c\}.
$$
Note that $S$ is the disjoint union of all these $S_{a,b}$.
Moreover, there is an easy-to-see bijection between the 
elements of $S_{a,b}$ and the elements of 
$\mathcal F(\Delta(\hatReg(\pi_{a,b})))$. Indeed,
observe that each of the $c\in S_{a,b}$ can be constructed
from a simplex in $\mathcal F(\Delta(\hatReg(\pi_{a,b})))$ by adding $\pi_{a,b}$.
Thus, we may apply induction hypothesis: for every $\pi_{a,b}$, 
there is an acyclic matching on
$\mathcal F(\Delta(\hatReg(\pi_{a,b})))$ with $s_O(\pi_{a,b})$ critical
simplices of dimension $n-4$.
In an obvious way, we may extend this acyclic matching to an
acyclic matching on $S_{a,b}$, leaving $s_O(\pi_{a,b})$ critical 
simplices of dimension $n-3$. 
This proves the recurrence stated in the Theorem.
\end{proof}

The recurrence in Theorem \ref{thm:numberofspheres} allows us to compute
the number of spheres in $\Delta(\hatReg(P))$ for any relevant finite 
poset $P$.
For a small poset $P$, this can be easily done by hand.
Playing with small examples yields a hypothesis that $s_O(P)=e(P)$;
the number of spheres is equal to the number of linear extensions
of $P$. However,
this is clearly not true, because for every $n$-element antichain $A_n$ one has
$s_O(A_n)=(n-1)!$ (Example \ref{ex:antichain}) while $e(A_n)=n!$. On the other
hand, it is possible to prove directly that things go well for a connected
poset: whenever $P$ is connected, $s_O(P)=e(P)$. This will be proved as
a corollary of the main result (Corollary \ref{coro:last}).

\section{Cyclic classes of linear extensions}

In this section, we prove that the number $s_O(P)$ is equal to the
number of cyclic classes of linear extensions of $P$.

Let $P$ be a finite nonempty poset with $n$ elements.
Let $f\colon P\to [0,n-1]_{\mathbb N}$ be a linear extension of $P$.
Consider the natural right action $(u,k)\mapsto u\oplus k$ 
of the cyclic group of order $n$
$(\Zn,\oplus)$ on itself. We write 
$\oplus_f\colon P\times\Zn\to P$ for the pullback
of this action by $f$. In other words,  
for all $x\in P$ and $k\in\Zn$,
$$
x\oplus_f k=f^{-1}(f(x)\oplus k).
$$
Analogously,
for $k\in\Zn$, we write $x\ominus_f k:=x\oplus_f(n-k)$.

Formally speaking, $\Zn$ plays a double role here: 
\begin{enumerate}
\item $\Zn$ is a set of natural numbers $[0,n-1]_{\mathbb N}$ being acted on.
Thus, among other things, $\Zn$ is a subposet of the poset $(\mathbb N,\leq)$.
\item $\Zn$ is a group that acts on a set of natural numbers $\Zn$
\end{enumerate}
This should not lead to any confusion.

Obviously, the $\oplus_f$ action of the element $1\in\Zn$ can be
represented by an oriented cycle digraph. The vertices of the digraph
are the elements of $P$, the edges are
\begin{multline*}
\{(x,x\oplus_f 1): x\in P\}\\
=\{(f^{-1}(0),f^{-1}(1)),\dots,(f^{-1}(n-2),f^{-1}(n-1)),
(f^{-1}(n-1),f^{-1}(0))\}.
\end{multline*}
We denote this digraph by $C(f,P)$.
As $\Zn$ is cyclic, the action of $1$, and thus the digraph,
determines the action of $\Zn$ on the set $P$.

\begin{definition}\label{def:ce}
Let $P$ be a finite poset, let $f,g$ be linear extensions of
$P$. 
We say that $f,g$ are {\em cyclically equivalent}, in symbols
$f\sim g$, if $\oplus_f=\oplus_g$. An equivalence class of
$\sim$ is called a {\em a cyclic class (of linear extensions) of $P$}. The number of cyclic classes
of $P$ is denoted by $e_C(P)$.
\end{definition}

Note that $f\sim g$ if and only if there are words $w,w'$ over $P$ such that
$f=ww'$ and $g=w'w$.

\begin{example}
\label{ex:discon}
Consider the disjoint sum of a chain of height $1$ and a 
one-element poset (Figure \ref{fig:discon}). This poset has
3 linear extensions giving rise to 2 cyclic classes.
\begin{figure}[t]
\includegraphics{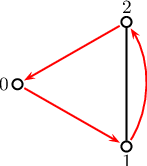}
\hskip 2em
\includegraphics{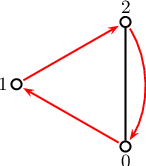}
\hskip 2em 
\includegraphics{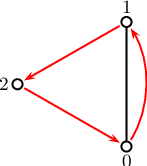}
\caption{Actions of $\mathbb Z_3$ on a 3-element poset}
\label{fig:discon}
\end{figure}
\end{example}
As we can see from the Example \ref{ex:discon}, it may well
happen that two distinct linear extensions of a finite poset determine the
same action. In this case, 
the $\sim$ relation is nontrivial and the number of cyclic classes is smaller than
the number of linear extensions, $e_C(P)<e(P)$.

\begin{proposition}
Let $P$ be a finite poset.
The following are equivalent.
\begin{itemize}
\item[(a)] $P$ is connected.
\item[(b)] For all linear extensions $f,g$ of $P$, $f\sim g$ implies that $f=g$.
\end{itemize}
\end{proposition}
\begin{proof}
(a)$\implies$(b): Let $f,g$ be linear extensions of $P$ such that 
$\oplus_f=\oplus_g$ and $f\neq g$. There are nonempty words $w_1,w_2$ over $P$
such that $f=w_1w_2$ and $g=w_2w_1$. Since $P$ is connected, then there are comparable elements
$x_1,x_2$ occurring in words $w_1,w_2$, respectively. 
However, it is clear that $f(x_1)<f(x_2)$ and
$g(x_1)>g(x_2)$, which imply $x_1<x_2$ and $x_1>x_2$, a contradiction.

(b)$\implies$(a):
If $P$ is not connected, then $P$ is a disjoint union of nonempty $P_1,P_2$ such that every pair
$x_1\in P_1$ and $x_2\in P_2$ is incomparable. Clearly, if $w_1$ is a linear extension
of $P_1$ and $w_2$ is a linear extension of $P_2$, then both $w_1w_2$ and $w_2w_1$ are 
cyclically equivalent linear extensions
of $P$.
\end{proof}
\begin{corollary}\label{coro:charconnected}
A finite poset $P$ is connected if and only if
$e(P)=e_C(P)$.
\end{corollary}

\begin{theorem}[Main result]\label{thm:mainresult}
Let $P$ be a finite poset with $n$ elements, $n\geq 3$. Then $\Delta(\hatReg(P))$ is
homotopy equivalent to a wedge of $e_C(P)$ spheres of dimension $n-3$. 
\end{theorem}

Our goal is to show that the number of cyclic classes of linear extensions is the same as the number of spheres
in $\Delta(\hatReg(P))$. To do this, we prove that the recurrence for $s_O(P)$  
from Theorem \ref{thm:numberofspheres} holds for $e_C(P)$ as well. Since
it is easy to check that $s_O(P)=e_C(P)$ for any 3-element poset $P$, the quantities 
must be equal.

To prove the recurrence for $e_C(P)$, we need to link cyclic classes of linear extensions of the poset $P$
with the cyclic classes of linear extensions of the posets $\pi_{a,b}$, where $\pi_{a,b}$ is an
order-preserving partition $P$.

Let us outline the schema of the proof of Theorem \ref{thm:mainresult}.
\begin{enumerate}
\item We prove that, for a fixed minimal
element $a$, there is a mapping $S_a$ from the set of all linear extensions of $P$ to the
disjoint union of sets of all linear extensions of all $\pi_{a,b}$, where $\pi_{a,b}$ is
order-preserving (Lemmas \ref{lemma:piab} and \ref{lemma:faislinear}).
\item We prove that this mapping is surjective (Lemma \ref{lemma:sasurjective}).
\item We prove that two linear extensions $f,g$ of $P$ are cyclically equivalent
if and only if their images $S_a(f),S_a(g)$ are cyclically equivalent
(Lemma \ref{lemma:cycequiv}).
\item These facts imply that $S_a$ determines a bijection from the set
of all cyclic classes of linear extensions of $P$ to the disjoint union of sets of all cyclic
extensions of all $\pi_{a,b}$, where $\pi_{a,b}$ is an order-preserving partition of $P$.
\item This implies that the $s_O(P)$ and $e_C(P)$ satisfy the same recurrence.
Since $s_O$ and $e_C$ are equal for 3-element posets, they are equal for any poset with
at least 3 elements.
\end{enumerate}

\begin{lemma}
\label{lemma:piab}
Let $P$ be a finite poset with $n$ elements, $n\geq 2$. Let $f$ be a
linear extension of $P$, let $a$ be a
minimal element of $P$. Then
$\pi_{a,a\oplus_f 1}$ is an order-preserving partition of $P$. 
\end{lemma}
\begin{proof}
If $f(a)<n-1$, then $f\bigl(a\oplus_f 1\bigr)=f(a)+1$, hence
$a\not\geq a\oplus_f 1$. Therefore,
either $a\leq a\oplus_f 1$ or $a\parallel a\oplus_f 1$.
If $a\parallel a\oplus_f 1$, then $\pi_{a,a\oplus_f 1}$ is order-preserving.
If $a\leq a\oplus_f 1$ then $\pi_{a,a\oplus_f 1}$ is order-preserving iff
$a\prec a\oplus_f 1$. Suppose that $a<b<a\oplus_f 1$.
Then $f(a)<f(b)<f\bigl(a\oplus_f 1\bigr)$, which contradicts
$f\bigl(a\oplus_f 1\bigr)=f(a)+1$.

If $f(a)=n-1$
(or, equivalently, $f(a\oplus_f 1)=0$), then $a$ is maximal.
Since we assume that $a$ is minimal, this implies
that $a$ is an isolated element, hence $a$ and $a\oplus_f 1$
are incomparable. This implies that $\pi_{a,a\oplus_f 1}$ is order-preserving.
\end{proof}

For a finite poset $P$ with $n\geq 2$ elements,
a linear extension $f$ of $P$, and a minimal element $a$ of $P$, 
let us define a mapping
$f_a\colon\pi_{a,a\oplus_f 1}\to[0,n-2]_{\mathbb N}$ by the rule
$$
f_a(B)=
\begin{cases}
f(x)&\text{if $B=\{x\}$ and $f(x)<f(a)$,}\\
\min\bigl(f(a),f\bigl(a\oplus_f 1\bigr)\bigr)&\text{if $B=\{a,a\oplus_f 1\}$,}\\
f(x)-1&\text{if $B=\{x\}$ and $f(x)>f(a)+1$.}
\end{cases}
$$
\begin{example}
\begin{figure}
\centering
\begin{minipage}{10em}
\vspace*{\fill}

\includegraphics{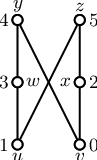}

\vspace*{\fill}
\end{minipage}
\begin{minipage}{10em}
\includegraphics{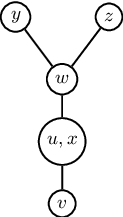}
\end{minipage}\\
\caption{Example 6}
\label{fig:shrink}
\end{figure}
Consider the 6-element poset $P$ from the left-hand side of Figure \ref{fig:shrink}. Let $g$
be a linear extension given by the number in the picture. Then the order-preserving
partition $\pi_{u,u\oplus_f 1}$ is equal to $\pi_{u,x}$, see the right hand side
of Figure \ref{fig:shrink}.

The values of the mapping $f_u\colon\pi_{u,x}\to[0,4]$ are computed as follows.
\begin{itemize}
\item Since $0=f(v)<f(u)=1$, $f_u(\{v\})=f(v)=0$.
\item $f(\{u,x\})=\min(f(u),f(x))=1$.
\item Since $3=f(w)>f(u)+1=2$, $f_u(\{w\})=f(w)-1=2$.
\item Similarly, $f(\{y\})=3$ and $f(\{z\})=4$.
\end{itemize}
\end{example}

\begin{lemma}\label{lemma:quotmap}
Let $\pi$ be an order-preserving partition of a poset $P$, let $q:P\to\pi$ be
the quotient map, given by, $q(x)=[x]_\pi$. Let $Q$ be a poset, let $g:\pi\to Q$
be a mapping of sets. Then $g$ is isotone if and only if $g\circ q$ is isotone.
\end{lemma}
\begin{proof}
Clearly, if $g$ is isotone, then $g\circ q$ is isotone.

Suppose that $g\circ q$ is isotone. Let $[x]_\pi,[y]_\pi\in\pi$ be such
that $[x]_\pi\leq[y]_\pi$.
By definition, there is a $\rho_\pi$-sequence $x=x_0,\dots,x_n=y$.
If $(x_{i-1},x_i)\in\rho_\pi$, then $q(x_{i-1})=q(x_{i})$, so $(g\circ q)(x_{i-1})=(g\circ q)(x_i)$.
If $x_{i-1}\leq x_i$ then $(g\circ q)(x_{i-1})\leq(g\circ q)(x_i)$, because $(g\circ q)$ is isotone.
Thus, 
$$
g([x]_\pi)=g([x_0]_\pi)=(g\circ q)(x_0)\leq (g\circ q)(x_n)=g([x_n]_\pi)=g([y]_\pi)
$$
and we see that $g$ is isotone.
\end{proof}

\begin{lemma}
\label{lemma:faislinear}
Let $P$ be a finite poset with $n\geq 2$ elements, let $f$ be a
linear extension of $P$, let $a$ be a minimal element of $P$.
Then $f_a$ is a linear extension of the poset $(\pi_{a,a\oplus_f 1},\leq)$.
\end{lemma}
\begin{proof}
It is obvious that $f_a$ is a bijection.

Let $\delta:[0,n-1]\to[0,n-2]$ be the degeneracy map corresponding to the
contraction of the edge $\{f(a),f(a)\oplus 1\}$ in the directed cycle $[0,n-1]$.
Let $q$ be the quotient map that maps every element of $P$ to its $\pi_{a,a\oplus_f 1}$-class.

Then
$f_a$ is the unique map that makes
$$
\xymatrix{
P\ar[r]^-f \ar[d]_{q}&[0,n-1]\ar[d]^{\delta}\\
\pi_a\ar[r]^-{f_a}&[0,n-2]
}
$$
commute. 

Since $q$ is a quotient map, $f_a$ is isotone if and only if $f_a\circ q=\delta\circ f$ is
(see Lemma \ref{lemma:quotmap}).
If $f(a)<n-1$, then $\delta$ is an isotone map, hence $\delta\circ f$ is isotone.
If $f(a)=n-1$, then $\delta$ is not isotone. However, in this case $a$ is an
isolated element of $P$ and it is easy to check that $\delta\circ f$ is isotone.
\end{proof}

Let $a$ be a minimal element of a finite poset $P$.
By the previous two lemmas, there is a mapping
$$
S_a\colon\mathcal \lin(P)\to\bigcup\{\lin(\pi_{a,b}) : \text{$\pi_{a,b}$ is order-preserving}\}
$$
given by $S_a(f):=f_a$.
In fact, this mapping is surjective, as shown by the following lemma.
\begin{lemma}
\label{lemma:sasurjective}
Let $P$ be a finite poset with $n\geq 2$ elements.
Let $a$ be a minimal element of $P$.
Let $b\in P$ be such that $\pi_{a,b}$ is an order-preserving partition. For every linear
extension $g$ of $\pi_{a,b}$ there is a linear extension $f$ of $P$ such that
$a\oplus_f 1=b$ and $f_a=g$.
\end{lemma}
\begin{proof}
The mapping $f\colon P\to[0,n-1]$ is given as follows:
$$
f(x)=
\begin{cases}
g(\{x\})&\text{if $g(\{x\})<g(\{a,b\})$,}\\
g(\{a,b\})&\text{if $x=a$,}\\
g(\{a,b\})+1&\text{if $x=b$,}\\
g(\{x\})+1&\text{if $g(\{x\})>g(\{a,b\})$.}
\end{cases}
$$
Obviously, $f$ is a bijection. We shall prove that $f$ is order-preserving.
Let $x,y\in P$ be such that $x<y$. Denote the class of $x$ in $\pi_{a,b}$ by $[x]$.
If $x=a$ and $y=b$, then $f(x)<f(y)$. Otherwise, $g([x])<g([y])$ and 
$$
f(x)\leq g([x])+1\leq g([y])\leq f(y).
$$
Thus, $f$ is a linear extension of $P$.
Clearly, 
$$
a\oplus_f 1=f^{-1}(f(a)\oplus 1)=f^{-1}(g(\{a,b\})+1)=f^{-1}(f(b))=b.
$$
The proof that $f_a=g$ is trivial and is thus omitted. 
\end{proof}
\begin{lemma}
\label{lemma:cycequiv}
Let $P$ be a finite poset with $n\geq 2$ elements.
Let $f,g$ be linear extensions of $P$,
let $a$ be a minimal element of $P$.
Then $f\sim g$ if and only if $f_a\sim g_a$.
\end{lemma}
\begin{proof}
Suppose that $\oplus_f=\oplus_g$.
This implies that $C(f,P)=C(g,P)$.
The mapping $f\mapsto f_a$, $g\mapsto g_a$ corresponds
to the contraction of the same edge 
$(a,a\oplus_f 1)=(a,a\oplus_g 1)$. Thus, 
$C(f_a,\pi_{a,a\oplus_f 1})=C(f_b,\pi_{a,a\oplus_g 1})$
and this implies that $\oplus_{f_a}=\oplus_{g_a}$.

Suppose that $\oplus_{f_a}=\oplus_{g_a}$.
The domains of equal maps must be the same, so
$\pi_{a,a\oplus_f 1}=\pi_{a,a\oplus_g 1}$.
Hence, $C(f_a,\pi_{a,a\oplus_f 1})=C(g_a,\pi_{a,a\oplus_g 1})$.
The digraph $C(f,P)$ arises from $C(f_a,\pi_{a,a\oplus_f 1})$
by an expansion of the vertex $\{a,a\oplus_f 1\}$.
Principally, there are two possible orientations of the new edge
between $a,a\oplus_f 1$.
However, only one of them gives us an oriented cycle.
Therefore, $C(f,P)$ is determined by 
$C(f_a,\pi_{a,a\oplus_f 1})$.
Similarly, $C(g,P)$ is determined by 
$C(g_a,\pi_{a,a\oplus_g 1})$.
\end{proof}

\begin{proof}[Proof of the main result]
It is easy to check that for any 3-element poset $P$,
$e_C(P)=s_O(P)$ (see Table \ref{tab:numberofspheres}). 

\begin{table}
\begin{center}
\begin{tabular}{|c|c|c|c|c|}
\hline
\raisebox{1em}{\includegraphics{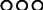}} &
\includegraphics{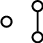} &
\includegraphics{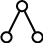} &
\includegraphics{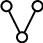} &
\includegraphics{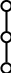} \\ 
\hline
2 & 2 & 2 & 2 & 1\\
\hline
\end{tabular}
\end{center}
\caption{$e_C(P)=s_O(P)$ for 3-element posets}
\label{tab:numberofspheres}
\end{table}

Let $a$ be a minimal element of a finite poset $P$,
$|P|>3$. By Lemmas \ref{lemma:sasurjective}~and~\ref{lemma:cycequiv},
$S_a$ determines a bijection
$$
S_a^\sim\colon (\lin(P)/\sim)\to
\bigcup\{\lin(\pi_{a,b})/\sim : \text{$\pi_{a,b}$ is order-preserving}\}
$$
given by $[f]_\sim\mapsto[f_a]_\sim$.
Since the union of the right-hand side is clearly disjoint,
this gives us the following recurrence
$$
e_C(P)=\sum_\text{$\pi_{a,b}$ is order-preserving} e_C(\pi_{a,b}).
$$

By induction and Theorem \ref{thm:numberofspheres}, $s_O(P)=e_C(P)$.
\end{proof}
\begin{corollary}\label{coro:last}
Let $P$ be a finite connected poset with $n$ elements, $n\geq 3$. Then $\Delta(\hatReg(P))$ is
homotopy equivalent to a wedge of $e(P)$ spheres of dimension $n-3$. 
\end{corollary}
\begin{proof}
This follows by Theorem \ref{thm:mainresult} and Corollary \ref{coro:charconnected}. 
\end{proof}

\noindent {\bf Acknowledgements}
{\em We are indebted to both anonymous referees for valuable comments and suggestions, especially 
to one of the referees for
the proofs of Lemmas \ref{lemma:sasurjective} and \ref{lemma:faislinear}; our original proofs 
were technically correct, but unnecessarily long and boring.
}

{\em
This research is supported by grants VEGA G-1/0297/11,G-2/0059/12 of M\v S SR,
Slovakia and by the Slovak Research and Development Agency under the contracts
APVV-0073-10, APVV-0178-11.
}

\end{document}